%% file: main_elastic_w.tex
\documentclass[10pt]{article}
\usepackage{graphicx} % Required for inserting images
\usepackage[letterpaper,top=2cm,bottom=2cm,left=3cm,right=3cm,marginparwidth=1.75cm]{geometry}
\newcommand{\RR}{{\mathbb{R}}}

\newcommand{\om}{\Omega}

\newcommand{\dive}{\text{\normalfont div}}

\input{packages}
\begin{document}

\maketitle

% \tableofcontents

\begin{abstract}
	\noindent
		We consider the inverse problem of determining an inclusion contained in an elastic body undergoing time-harmonic oscillations at fixed frequency by local Cauchy data. Both the body and the inclusion
    	are made by different homogeneous linearly elastic isotropic materials, with piecewise constant mass densities. Under mild a priori regularity assumptions on the unknown inclusion and with no spectral hypothesis on the frequency, we establish a logarithmic-type stability estimate by local Cauchy data. We also provide a stability result in terms of the so local Dirichlet to Neumann map for sufficiently small frequency. 
\end{abstract}

\section{Introduction}

In this work we study the inverse problem of determining an inclusion $D$ contained in an elastic body $\Omega$ from the time-harmonic local data representing  the oscillation amplitudes and the corresponding boundary tractions collected on a portion  $\Sigma$ of the boundary.

For a prescribed boundary displacement $f \in H^{1/2}(\partial \Omega)$, the small time-harmonic oscillations with frequency $\omega$ and amplitude $u(x) \in H^1(\Omega)$ of the continuum are governed by the following boundary value problem
\begin{equation}
	\label{eqn:BVP}
	\begin{cases}
		\Div(( \C + (\C^{I} - \C)\chi_D(x)   ) \nh u) + \omega^{2}( \rho + (\ri - \rho)\chi_D(x)  ) u = 0 &\text{in }\Omega,\\
		u = f &\text{on }\partial\Omega,
	\end{cases}
\end{equation}
where $\Omega$ is a bounded domain in $\mathbb{R}^3$ and $D$ is an open set strictly contained in $\Omega$ with $C^{1,\alpha}$ boundary. The subdomain $\Omega \setminus D$ and the inclusion $D$ consist of different isotropic and homogeneous elastic materials, with elasticity tensors $\C$, $\C^I$ determined by the Lamé parameters $\mu,\lambda$ and $\mu^I, \lambda^I$, respectively, and such that $\C$, $\C^I$ are strongly convex in $\Omega$. The mass density is piecewise constant and equal to $\rho$, $\rho^I$ in $\Omega \setminus D$ and in $D$, respectively. In \eqref{eqn:BVP}, $\chi_D$ is the characteristic function of $D$ and the strain tensor is denoted by  $\nh u = (\nabla u +(\nabla u)^{T})/2$.

We denote by $\mathcal{C}_D^{\Sigma}(\Sigma)$ the set of all possible Cauchy data $(u|_{\Sigma}, (\C_D\nh u) n|_{\Sigma})$ associated to the problem, where $n$ is the outer unit normal to $\Omega$ . The main goal of this research is to establish a stability estimate for the determination of the inclusion $D$ in terms of the local Cauchy data $\mathcal{C}_D^{\Sigma}(\Sigma)$.

Let us observe that for $\omega=0$, namely when one deals with the static case, the direct problem is well posed and one can define the so called local Dirichlet to Neumann map 
\begin{eqnarray}\label{eqn: DtoN intro}
	\Lambda_D^{\Sigma}: H^{1/2}_{00}(\Sigma) \rightarrow
	H^{-1/2}_{00}(\Sigma),\\
	f \mapsto (\C_D\nh u) n|_{\Sigma}
\end{eqnarray}
which assigns to each displacement prescribed on $\Sigma$ the corresponding traction (see Section \ref{sec:local-Cauchy} for precise definitions of the trace spaces). Of course, when no spectrum conditions on $\omega$ are given, the existence of such operator is not guaranteed. 

In our context, the use of Cauchy data in place of the local Dirichlet to Neumann map appears more appropriate, as it enables to circumvent spectral-type frequency restrictions, thereby yielding a result that is suitable for applications.
We recall that the scalar counterpart of this problem, namely the determination of an unknown inclusion  for the Schr\"{o}dinger equation, has been discussed by Foschiatti and Sincich in \cite{F-S} while, to the best of our knownledge, stability estimates for the time harmonic elastic wave equation allowing resonance condition have not yet been provided. 

The physical problem described by \eqref{eqn:BVP} is of significant relevance in diagnostic applications in seismology and geophysics. As observed in \cite{Beretta2017}, a notable application arises in seismology. Established studies dating back to the 1980s show that methods based on measurements of the Dirichlet-to-Neumann map are effective in reconstructing the mechanical properties of the soil and the mass density profile at depth. We refer to \cite{Baeten-Tesi-PhD} for an early overview of potential geophysical applications and the experimental methodologies employed. These applications confirm that the assumption of piecewise constant coefficients is quite realistic and describes concrete situations in which, for instance, a significant discontinuity in the elastic coefficients or in the mass density occurs across the interface of the inclusion.

Inverse problems in linear elasticity have been extensively studied over the past decades, both in the static (e.g., $\omega=0$ in \eqref{eqn:BVP}) and in the dynamic regime. One line of research has focused, for instance, on the determination of the Lamé moduli from the Dirichlet-to-Neumann map for $\omega=0$, in analogy with the corresponding problem for a conducting body. Fundamental contributions in this direction can be traced back, among others, to the works of Nakamura and Uhlmann \cite{N-U1}, \cite{N-U2}, \cite{N-U3}, Eskin and Ralston \cite{E-R}, and Imanuvilov, Uhlmann and Yamamoto \cite{I-U-Y}. These results typically hold for sufficiently regular coefficients satisfying suitable a priori assumptions. A Lipschitz stability result in the case of piecewise constant Lamé moduli, with unknown coefficients but with a known decomposition of the domain, has been obtained by Beretta, Francini and Vessella \cite{B-F-V-2014}.

More specifically related to the determination of inclusions, an alternative approach has been based on the identification of the unknown boundary of the inclusion, namely the interface $\partial D$, from boundary measurements. In the static case $\omega=0$, Morassi and Rosset obtained log-log type stability estimates from a single pair of Cauchy data on $\partial \Omega$ in the extreme cases of a cavity \cite{M-R1} and of a rigid inclusion \cite{M-R2}. Alessandrini et al. \cite{AleDiCMorRos14} proved a logarithmic stability estimate for the identification of an elastic inclusion with regular boundary embedded in an elastic body from the Dirichlet-to-Neumann map. The result holds under the assumption that the Lamé moduli of the inclusion are constant and different from those of the surrounding material. The extension of the validity of the estimate to variable coefficients, sufficiently regular both outside and inside the inclusion, has been established by Morassi and Rosset in \cite{M-R3}.

Modal measurements were used by Ammari et al. \cite{A-B-F-K-L-2010} to determine small changes in the interface $\partial D$ of an elastic inclusion, when both the inclusion and the surrounding body are made by homogeneous linearly elastic isotropic material. A reconstruction procedure based on measurements of one eigenvalue and the corresponding eigenfunctions was proposed using  an optimization approach.

For the inverse problem of detecting an inclusion from the time-harmonic elastic waves as \eqref{eqn:BVP}, the presence of frequency-dependent effects and possible resonances introduces additional analytical challenges with respect to the static case. Therefore, despite the clear physical meaning and practical impact of the inverse problem, only partial results are currently available. Reconstruction approaches based on sampling and factorization methods have been proposed by Cakoni, Colton and Haddar \cite{CakoniColtonHaddar2016}, although mostly in the context of scattering data rather than boundary measurements. Beretta et al. \cite{Beretta2017} consider the determination of the Lamé parameters and the mass density in three dimensions from the local time-harmonic Dirichlet-to-Neumann map. The authors proved a Lipschitz stability estimate when the Lamé parameters and the mass density are piecewise constant on a given finite partition of the domain. 

More recently, Eberle-Blick and Pohjola \cite{E-B-P-2024} extended the monotonicity method (see also \cite{H-P-S-2019}, \cite{E-H-2021}) to the time-harmonic elastic problem \eqref{eqn:BVP} when the Neumann-to-Dirichlet map is given. This is a shape reconstruction method that allows one to obtain an approximation of the inclusion domain $D$, provided that $\Omega \setminus D$ is connected. We refer to \cite{E-B-P-2024} for an updated state of the art on shape reconstruction methods for the linear elasticity system. It should be noted that the reconstruction methodology proposed in \cite{E-B-P-2024} works for specific forms of coefficient perturbations, namely when the Lamé parameters increase and the mass density decreases in the inclusion $D$, or when the Lamé parameters remain constant and the mass density increases.

As main result of the paper, we provide in Theorem \ref{thm: stability2} a  stability estimate for the unknown inclusion $D$ in terms of local Cauchy data on $\Sigma$ which works also in resonant regime.  The inverse problem is, however, highly ill-conditioned and in fact the stability estimate is of logarithmic type. As in the approach used by Alessadrini, de Hoop, Gaburro and Sincich \cite{A-dH-G-S}, we have chosen to express the errors on the boundary data in terms of the so called angle (or distance) between the spaces of Cauchy data interpreted as subspaces of a suitable Hilbert space. 

An alternative approach to stability is based on the representation of the data throughout the local Dirichlet-to-Neumann map $\Lambda_D$ given on $\Sigma$. In this case, assuming the wave frequency $\omega$ sufficiently small to avoid the onset of resonance phenomena, we prove an analogous logarithmic stability result.  The precise statement is presented in Theorem \ref{thm: stability}. Our proof is inspired by the techniques developed in the papers by Alessandrini et al. \cite{AleDiCMorRos14} (which consider the static case of \eqref{eqn:BVP} with $\omega=0$) and by Beretta et al. \cite{Beretta2017}.

%\textcolor{red}{(to be completed)}.

\section{Definitions and a priori assumptions}

\subsection{Notation and Definitions}

For $x\in{\R}^3$, let us denote $x=(x',x_3)$, where $x'=(x_1,x_2)\in{\R}^{2}$,
$x_3\in{\R}$. Let us denote by $\R^{3}_+= \{x\in \R^3 \ | \ x_3>0 \}$ and
$\R^3_-= \{x\in \R^3 \ | \ x_3<0 \}$. For $x \in \R^3$, $r>0$, we shall use the
following notation for balls and cylinders.
\begin{equation*}
    B_r(x)=\{y \in \R^3 \ \mid \ |y-x|<r\}, \quad B_r=B_r(0),
\end{equation*}
\begin{equation*}
	B_r^+(x)=\{ y \in \R^3 \ \mid \ |y-x|<r, \ y_3>0 \}, \quad B_r^+=B_r^+(0),
\end{equation*}
\begin{equation*}
	B_r^-(x)=\{  y \in \R^3 \ \mid \ |y-x|<r,  \ y_3<0 \}, \quad B_r^-=B_r^-(0),
\end{equation*}
\begin{equation*}
    B'_r(x')=\{y' \in \R^2 \ \mid \ |y'-x'|<r\}, \quad B'_r=B'_r(0),
\end{equation*}
\begin{equation*}
    Q_{a,b}(x) = \{ (y', y_3) \ \mid \ |y'-x'|<a, \ |y_3-x_3|<b \},
    \quad Q_{a,b}=Q_{a,b}(0),
\end{equation*}
\begin{equation*}
    Q_{a,b}(x)^+ = \{ (y', y_3) \ \mid \ |y'-x'|<a, \ 0<y_3-x_3<b \},
    \quad Q_{a,b}^+=Q_{a,b}^+(0).
\end{equation*}

\begin{definition}
  \label{def:2.1} (${C}^{k,\alpha}$ regularity)
Let $E$ be a bounded domain in ${\R}^{3}$. Given $k$,
$\alpha$, with $k\in \N$, $0<\alpha\leq 1$, we say that $E$
is of \textit{class ${C}^{k,\alpha}$ with
constants $r_0$, $M_{0}>0$}, if, for any $P \in \partial E$, there
exists a rigid transformation of coordinates under which we have
$P=0$ and
\begin{equation*}
  E \cap B_{r_{0}}(0)=\{x \in B_{r_{0}}(0)\ \mid \ x_{3}>\varphi(x')
  \},
\end{equation*}
where $\varphi$ is a ${C}^{k,\alpha}$ function on $B'_{r_{0}}$
satisfying
\begin{equation*}
\varphi(0)=0,
\end{equation*}
\begin{equation*}
\nabla \varphi (0)=0, \quad \hbox{when } k \geq 1,
\end{equation*}
\begin{equation*}
\|\varphi\|_{{C}^{k,\alpha}(B'_{r_{0}}(0))} \leq M_{0}r_{0}.
\end{equation*}

\medskip
\noindent When $k=0$, $\alpha=1$, we also say that $E$ is of
\textit{Lipschitz class with constants $r_{0}$, $M_{0}$}.
\end{definition}

\begin{remark}
  \label{rem:2.1}
  In this paper, we use the convention to normalize all norms in such a way that their
  terms are dimensionally homogeneous with the $L^{\infty}$ norm.  In Definition \ref{def:2.1}, for example, the $C^{k,\alpha}$ norm is intended as follows:
\begin{equation*}
  \|\varphi\|_{{C}^{k,\alpha}(B'_{r_{0}}(0))} =
  \sum_{i=0}^k r_0^i
  \|\nabla^i\varphi\|_{{L}^{\infty}(B'_{r_{0}}(0))}+
  r_0^{k+\alpha}|\nabla^k\varphi|_{\alpha, B'_{r_0}(0)},
\end{equation*}
where
\begin{equation*}
|\nabla^k\varphi|_{\alpha, B'_{r_0}(0)}= \sup_
{\overset{\scriptstyle x', \ y'\in B'_{r_0}(0)}{\scriptstyle
x'\neq y'}} \frac{|\nabla^k\varphi(x')-\nabla^k\varphi(y')|}
{|x'-y'|^\alpha}.
\end{equation*}

Similarly, for a vector function $u\in H^{1}(\Omega, \R^{3})$, we set
\begin{equation*}
\|u\|_{H^1(\Omega, \R^3)}=r_0^{-\frac{3}{2}}\left(\int_\Omega u^2
+r_0^2\int_\Omega|\nabla u|^2\right)^{\frac{1}{2}},
\end{equation*}
and a similar convention holds for trace norms such as
$\|\cdot\|_{H^{\frac{1}{2}}(\partial\Omega, \R^3)}$,
$\|\cdot\|_{H^{-\frac{1}{2}}(\partial\Omega, \R^3)}$.
\end{remark}

Let us recall that the Hausdorff distance $\dH(D,\tilde D)$ between two bounded closed sets $D, \tilde D \subset \R^3$, is defined as
\begin{equation*}
  \label{eq:Hausdorff}
  \dH(D,\tilde D) = \max\{\max_{x\in D} d(x,\tilde D), \max_{x\in \tilde D} d(x,D)\} \ .
\end{equation*}
%where $d(w,D) := \inf\limits_{y\in D} |w-y|$ for $w\in \tilde D$.

\medskip

Let $\M^{m\times n}$ be the space of $m\times n$ real valued matrices and denote by ${\cal L} (X, Y)$ the space of bounded linear operators between the Banach spaces $X$ and $Y$. In our case we consider $m=n=3$, and denote $\M^{3\times 3}$ with $\M^{3}$.

For any $3\times 3$ matrices $A$ and $B$, and for every $\C\in{\cal L} ({\M}^{3}, {\M}^{3})$, the component-wise action of the tensor $\C$ on $\M^{3}$ is
\begin{equation}
  \label{eq: tensor1}
  ({\C}A)_{ij} = \sum_{k,l=1}^{3} C_{ijkl}A_{kl},
\end{equation}
the scalar product between the matrices $A$ and $B$ is defined as follows
\begin{equation}
  \label{eq: tensor2}
  A \cdot B = \sum_{i,j=1}^{3} A_{ij}B_{ij},
\end{equation}
and finally, the Frobenius norm of a matrix $A$ is defined as 
\begin{equation}
  \label{eq: tensor3}
  |A|= (A \cdot A)^{\frac {1} {2}},
\end{equation}
where $C_{ijkl}$, $A_{ij}$ and $B_{ij}$ are the entries of $\C$,
$A$ and $B$ respectively.

For every pair of real $3$-vectors $v$ and $w$, we denote by $v
\otimes w$ the $3 \times 3$ matrix with entries
\begin{equation}
  \label{eq:diade}
  (v \otimes w)_{ij} = v_i w_j, \quad i,j=1,2,3.
\end{equation}
Finally, given a $3\times 3$ matrix $A\in \M^{3}$, we denote $\widehat A = \frac{1}{2}(A + A^{T})$, where $A^{T}$ is the transpose matrix of $A$.

\subsection{A priori information}\label{Subsec: a-priori}

\begin{enumerate}[label = (\roman*)]
    \item \textit{Domain}. Let $\Omega$ be a bounded domain in $\R^3$ such that
\begin{equation}\label{eqn: domainstart}
  \R^3\setminus \overline{\Omega}\ \hbox{is connected},
\end{equation}
\begin{equation}\label{diametro}
  \mbox{diam} (\Omega)\leq M_1 r_0,
\end{equation}
\begin{equation}\label{eqn: domainend}
  \Omega \ \hbox{is of class } C^{1,\alpha}, \ \hbox{with
  constants } \ r_0, \ M_0,
\end{equation}
where $r_0$, $M_0$, $M_1$ are given positive constants, and
$0<\alpha < 1$.
\item \textit{Boundary portion}. Let $\Sigma$ be a nonempty open portion of the boundary $\p\Omega$ with size at least $r_0$, namely, we assume that there exists a point $P_{\Sigma}\in \Sigma$ such that
\begin{equation}\label{eqn: ipSigma}
    \text{dist}(P_{\Sigma},\p\Omega\setminus \Sigma) \geq r_0.
\end{equation}
\item \textit{Inclusion}. Let $D$ be a domain contained in $\Omega$ satisfying
\begin{equation}\label{eqn: inclstart}
  \R^3 \setminus \overline{D} \ \hbox{is connected},
\end{equation}
\begin{equation}\label{eqn: inclend}
    D \ \hbox{is of class } C^{1,\alpha},\ \hbox{with constants } \ r_0, \ M_0,
\end{equation}
where $0<\alpha< 1$.
\item \textit{Material properties}. The body $\Omega$ is assumed to be made of linearly elastic, isotropic and homogeneous material, with elastic tensor $\C$ of components
\begin{equation}\label{eqn: definition_C}
    C_{ijkl} = \lambda \delta_{ij}\delta_{kl} + \mu(\delta_{ki}\delta_{lj} +\delta_{li}\delta_{kj}),
\end{equation}
where $\delta_{ij}$ is the Delta Kronecker. The constant Lamé moduli, $\lambda$ and $\mu$, satisfy the strong convexity conditions
\begin{equation}\label{eqn: LameCond}
    \mu\geq \alpha_0,\quad 2\mu + 3\lambda\geq \beta_0,
\end{equation}
where $\alpha_0, \beta_0 >0$ are given constants. Moreover, we assume the following upper bounds for the Lamé parameters
\begin{equation}\label{eqn: LameUpperBound}
    \mu\leq \overline{\mu},\quad \lambda\leq \overline{\lambda},
\end{equation}
where $\overline{\mu} ,\overline{\lambda}$ are given real numbers. 

Notice that \eqref{eqn: LameCond} is equivalent to
\begin{equation}
    \C A \cdot A \geq \xi_0 |A|^{2},
\end{equation}
for any  $3\times 3$ symmetric matrix $A$, with $\xi_0 = \min\{2\alpha_0, \beta_0\}$.

Similarly, the elastic inclusion $D$ is made of isotropic, homogeneous material having elastic tensor $\C^{I}$ with Lamé moduli $\lambda^{I}, \mu^{I}$ satisfying conditions \eqref{eqn: LameCond}, \eqref{eqn: LameUpperBound}, and such that
\begin{equation}\label{eqn: jump}
    (\lambda - \lambda^{I})^{2} + (\mu- \mu^{I})^{2} \geq \eta_0^{2}>0,
\end{equation}
for a given constant $\eta_0>0$.

We assume that the density $\rho_D$ is of the form
\begin{equation}\label{eqn: definition_rho}
    \rho_D(x) = \rho + (\ri - \rho)\chi_D(x),
\end{equation}
where the constants $\rho, \ri$ satisfy the following bounds
\begin{equation}\label{eqn: rho_bounds}
    0<\gamma_0 \leq \rho, \quad \ri \leq \gamma_0^{-1}
\end{equation}
and that the frequency $\omega$ satisfies
\begin{equation}\label{eqn: omega_bounds}
	0<\omega \leq \frac{\delta_0}{r_0}.
\end{equation}
Let us denote 
\begin{equation}\label{elastictensorwithinclusion}
    \C_D(x) = \C + (\C^{I} - \C)\chi_D(x), \quad x\in \Omega,
\end{equation}
the tensor in the elastic body containing the inclusion.
\end{enumerate}

In the sequel, we will refer to the set of parameters
\begin{equation*} 
	\quad M_1,\quad M_0, \quad \alpha, \quad \alpha_0, \quad \beta_0, \quad \bar\lambda, \quad \bar\mu, \quad \eta_0, \quad \gamma_0, \quad \delta_0
\end{equation*}
as the {\it a priori data}, whereas the dimensional parameter $r_0$ shall appear explicitly in all our estimates. In the sequel we shall denote by $C$ a positive constant which may change from line to line and only depends on the a priori data.

\section{The local Cauchy Data and the main stability result}
\label{sec:local-Cauchy}

Let $\Sigma$ be the open portion of the boundary $\partial \Omega$ as in (ii),  Section \ref{Subsec: a-priori}. Consider the trace space
\[
H^{1/2}_{co}(\Sigma) = \{f\in H^{1/2}(\p\Omega)\,|\,\text{supp}(f) \subset \Sigma\},
\]
and consider its closure under the $H^{1/2}$-norm:
\[
H^{1/2}_{00}(\Sigma) = \overline{H^{1/2}_{co}(\Sigma)}^{H^{1/2}(\p\Omega)}.
\]
We then consider its dual space defined as follows
\[
H^{-1/2}_{00}(\Sigma) = \{ g\in H^{-1/2}(\p\Omega)\,|\,\langle g,\varphi \rangle = 0,\,\,\forall \varphi \in H^{1/2}_{00}(\Sigma) \},
\]
where $\langle \cdot,\cdot \rangle$ represents the duality between the spaces $H^{1/2}(\p\Omega)$ and $H^{-1/2}(\p\Omega)$,  based on the inner product on $L^{2}(\p\Omega)$
\[
\langle g,\varphi \rangle = \int_{\p\Omega} g\cdot \varphi.
\]

%Let $D\subset \Omega$ be an inclusion. 

\begin{definition}
	The set of Cauchy data on $\Sigma$ associated to the inclusion $D$ is defined as the set
	\[
	\begin{split}
		\mathcal{C}^{\Sigma}_D := \{(&f,g)\in H^{1/2}_{00}(\Sigma)\times H^{-1/2}(\p\Omega)\,:\\
		&\exists u\in H^{1}(\Omega) \text{ weak solution of }\\
		&\Div(\C_D \nh u) + \rho_D \omega^{2} u = 0 \text{ in }\Omega,\\
		&u|_{\p\Omega} = f, \quad (\C\nh u) n|_{\p\Omega} = g \}
	\end{split}
	\]
\end{definition}

Consider the dual of the trace space of functions with compact support in $\p\Omega\setminus\overline{\Sigma}$:
\[
H^{-1/2}_{00}(\p\Omega\setminus\overline{\Sigma}):=\left\{ \psi\in H^{-1/2}(\p\Omega)\,:\, \langle \psi,\varphi\rangle =0, \quad \forall\: \varphi\in H^{1/2}_{00}(\Sigma)\,\right\}.
\]

We define $H^{1/2}(\p\Omega)\vert_{\Sigma}$ and $H^{-1/2}(\p\Omega)\vert_{\Sigma}$ as the restrictions to $\Sigma$ of the trace spaces $H^{1/2}(\p\Omega)$ and $H^{-1/2}(\p\Omega)$, respectively. These trace spaces can be defined equivalently as quotient spaces by the following relation:
\[
\varphi_1 \sim \varphi_2 \qquad \iff \qquad \varphi_1 - \varphi_2 \in H^{1/2}_{00}(\p\Omega\setminus\overline{\Sigma}),
\]
so that
\[
H^{1/2}(\p\Omega)\vert_{\Sigma} = H^{1/2}(\p\Omega) \slash \sim\,\, = H^{1/2}(\p\Omega) \slash H^{1/2}_{00}(\p\Omega\setminus\overline{\Sigma}).
\]
Similarly,
\[
H^{-1/2}(\p\Omega)\vert_{\Sigma} = H^{-1/2}(\p\Omega)\slash H^{-1/2}_{00}(\p\Omega\setminus\overline{\Sigma}).
\]

\begin{definition}\label{def5: localCD}
	The \emph{local Cauchy data} on $\Sigma$  associated with the inclusion $D$, with the first component vanishing at $\p\Omega\setminus\overline{\Sigma}$, is defined as
	\begin{equation}
		\begin{split}
			\mathcal{C}_D^{\Sigma}(\Sigma)= \Big\{ &(f,g)\in H^{1/2}_{00}(\Sigma)\times H^{-1/2}(\p\Omega)\vert_{\Sigma}\ :\
			\exists\: u\in H^1(\Omega)\: \text{weak solution of}\\
			&\Div(\C_D\:\nh u) + \rho \omega^{2}\:u = 0 \qquad \text{in}\:\Omega,\\
			&u|_{\p\Omega} = f,\notag\\
			&\langle (\C_D\nh u) n\vert_{\p\Omega},\varphi\rangle= \langle g, \varphi \rangle \qquad \text{for all } \varphi \in H^{1/2}_{00}(\Sigma) \Big\}.
		\end{split}
	\end{equation}
\end{definition}
It is important to note that $\mathcal{C}^{\Sigma}_D(\Sigma)$ is a subspace of the product space 
\begin{equation}\label{eqn5: productH}
	\mathcal{H} := H^{1/2}_{00}(\Sigma)\times H^{-1/2}(\p\Omega)\vert_{\Sigma},
\end{equation}
which is a Hilbert space with norm
\begin{equation}\label{eqn5: normH}
	\|(f,g)\|_{\mathcal{H}}= \Big( \|f\|^2_{H^{1/2}_{00}(\Sigma)} +r_0^2 \|g\|^2_{H^{-1/2}(\p\Omega)\vert_{\Sigma}}\Big)^{1/2}\qquad \textnormal{for each}\:(f,g)\in\mathcal{H}.
\end{equation}
Given two inclusions $D_1$, $D_2$ satisfying \eqref{eqn: inclstart}--	\eqref{eqn: inclend} , we denote by $\mathcal{C}_i$ the local Cauchy data $\mathcal{C}^{\Sigma}_{D_i}$ associated with the inclusion $D_i$, for $i=1,2$. 

To compare two local Cauchy data, we use the definition of the distance between closed subspaces of a Hilbert space. Given $\mathcal{F}$ and $\mathcal{G}$, two subspaces of a Hilbert space, the \emph{distance} or \emph{aperture} between them is given by the following formula:
\begin{equation}\label{eqn5: aperture}
	d(\F,\G)=\max\Big\{ \sup_{h\in\G, h\neq 0} \inf_{k\in \F} \frac{\|h-k\|}{\|h\|}, \sup_{k\in\F, k\neq 0} \inf_{h\in \G} \frac{\|h-k\|}{\|k\|} \Big\}.
\end{equation}
It is well known \cite[Corollary 2.13]{KJA2010} that if $d(\F,\G)<1$, then the two quantities in \eqref{eqn5: aperture} coincide. Hence, we can assume that
\[
d(\F,\G) = \sup_{h\in\G, h\neq 0} \inf_{k\in \F} \frac{\|h-k\|}{\|h\|}.
\]
In our context, the distance between two local Cauchy data $\mathcal{C}_1$ and $\mathcal{C}_2$ is considered smaller than 1, then we can assume that it has the form
\begin{equation}\label{eqn5: cauchy}
	d(\mathcal{C}_1,\mathcal{C}_2)= \sup_{(f_1,g_1)\in \mathcal{C}_1\setminus\{(0,0)\}} \inf_{(f_2,g_2)\in \mathcal{C}_2} \frac{\|(f_1,g_1)-(f_2,g_2)\|_{\mathcal{H}}}{\|(f_1,g_1)\|_{\mathcal{H}}}.
\end{equation}

Let $u_i\in H^{1}(\Omega)$ for $i=1,2$ be a weak solution to
\begin{equation}
	\Div(\C_{D_i}\nh u_i) + \rho_{D_i} \omega^2 u_i = 0, \qquad \text{in }\Omega.
\end{equation}
The weak formulation reads as follows:
\begin{equation}\label{eqn: weakFormUi}
	\int_{\Omega} \C_{D_i}\nh u_i\cdot \nh \psi - \int_{\Omega} \rho_{D_i}\omega^{2} u_i \psi = \langle ( \C_{D_i} \nh u_i) n, \psi|_{\partial \Omega}\rangle,\qquad\text{for all }\psi\in H^{1}(\Omega).
\end{equation}
By choosing $i=1,2$  and $\psi=u_2, u_1$ in \eqref{eqn: weakFormUi}, and by subtracting the two equalities, we derive
\begin{equation}\label{eqn: AlIdentity}
	\begin{split}
		\int_{\Omega} &(\C_{D_2} - \C_{D_1})\nh u_1 \cdot \nh u_2 - \int_{\Omega} (\rho_{D_2} - \rho_{D_1})\omega^{2} u_1 u_2 \\
		&= \langle (\C_{D_2} \nh u_2) n|_{\p\Omega}, u_1|_{\p\Omega} \rangle - \langle (\C_{D_1} \nh u_1) n|_{\p\Omega}, u_2|_{\p\Omega} \rangle.
	\end{split}
\end{equation}
Let $v_i$ for $i=1,2$ be a weak solution to
\begin{equation}
	\Div(\C_{D_i}\nh v_i) + \rho_{D_i} \omega^{2} v_i = 0, \qquad \text{in }\Omega.
\end{equation}
The corresponding weak formulation is
\begin{equation}\label{eqn: weakFormVi}
	\int_{\Omega} \C_{D_i}\nh v_i\cdot \nh \psi - \int_{\Omega} \rho_{D_i}\omega^{2} v_i \psi = \langle (\C_{D_i} \nh v_i)n, \psi|_{\partial \Omega}\rangle,\qquad\text{for all }\psi\in H^{1}(\Omega).
\end{equation}
Choosing $i=2$ and $\psi=u_2$, we derive, using the symmetry of $\C_2$,
\[
\begin{split}
	\langle (\C_{D_2} \nh v_2) n, u_2|_{\partial \Omega}\rangle 
	&= \int_{\Omega} \C_{D_2}\nh v_2\cdot \nh u_2 - \int_{\Omega} \rho_{D_2}\omega^{2} v_2 u_2\\
	&= \int_{\Omega} \C_{D_2}\nh u_2\cdot \nh v_2 - \int_{\Omega} \rho_{D_2}\omega^{2} u_2 v_2\\
	&= \langle (\C_{D_2} \nh u_2) n, v_2|_{\partial \Omega}\rangle.
\end{split}
\]
Hence
\begin{equation}\label{eqn: weakBoundary}
	\langle (\C_{D_2} \nh v_2) n, u_2|_{\partial \Omega}\rangle = \langle (\C_{D_2} \nh u_2) n, v_2|_{\partial \Omega}\rangle.
\end{equation}
By summing \eqref{eqn: AlIdentity} and \eqref{eqn: weakBoundary}, we derive
\begin{equation}\label{eqn: AlIdentity2}
	\begin{split}
		\int_{\Omega} &(\C_{D_2} - \C_{D_1})\nh u_1 \cdot \nh u_2 - \int_{\Omega} (\rho_{D_2} - \rho_{D_1})\omega^{2} u_1 u_2 \\
		&= \left\langle (\C_{D_2} \nh u_2) n|_{\p\Omega}, (u_1-v_2)|_{\p\Omega} \right\rangle - \left\langle \Big( (\C_{D_1} \nh u_1) n|_{\p\Omega} - (\C_{D_2}\nh v_2) n|_{\p\Omega}\Big), u_2|_{\p\Omega} \right\rangle.
	\end{split}
\end{equation}
Our goal is now to estimate from above \eqref{eqn: AlIdentity2} in terms of the local Cauchy data sets.

By applying the Cauchy-Schwarz inequality and the definition of aperture, we derive
\begin{equation}\label{stimaCauchy}
	\begin{split}
		\Big|\int_{\Omega} &(\C_{D_2} - \C_{D_1})\nh u_1 \cdot \nh u_2 - \int_{\Omega} (\rho_{D_2} - \rho_{D_1})\omega^{2} u_1 u_2 \Big|\\
		&\leq d(\CC_1,\CC_2) \|(u_1,(\C_1\nh u_1)n)\|_{\mathcal{H}}
		\cdot \|(u_2,(\C_2\nh u_2) n)\|_{\mathcal{H}}.
	\end{split}
\end{equation}
Indeed,
\begin{equation}\label{AleIdCD}
	\begin{split}
		\Big|\int_{\Omega} &(\C_{D_2} - \C_{D_1})\nh u_1 \cdot \nh u_2 - \int_{\Omega} (\rho_{D_2} - \rho_{D_1})\omega^{2} u_1 u_2 \Big|\\
		&\leq \Big|\langle (\C_{D_2}\nh u_2)n|_{\partial \Omega}, (u_1 - v_2)|_{\p\Omega}\rangle \Big| + \Big|\langle ((\C_{D_1}\nh u_1) n - (\C_{D_2}\nh v_2) n)|_{\partial \Omega}, u_2|_{\p\Omega}\rangle \Big|\\
		&\leq \|(\C_{D_2}\nh u_2) n \| \cdot \|u_1 - v_2\| + \|(\C_{D_1}\nh u_1) n - (\C_{D_2}\nh v_2) n \| \cdot \|u_2\|\\
		&\leq \|(u_2, (\C_{D_2}\nh u_2)n)\| \cdot \Big( \|u_1 - v_2\| + \|(\C_{D_1}\nh u_1)n - (\C_{D_2}\nh v_2) n \| \Big)\\
		&= \|(u_2,(\C_{D_2}\nh u_2) n)\|\cdot \|(u_1,(\C_{D_1}\nh u_1) n)- (v_2,(\C_{D_2}\nh v_2) n)\|\\
		&\leq d(\CC_1,\CC_2) \|(u_1,(\C_1\nh u_1) n)\|_{\mathcal{H}}
		\cdot \|(u_2,(\C_2\nh u_2) n)\|_{\mathcal{H}}.
	\end{split}
\end{equation}

Our main result, Theorem \ref{thm: stability2}, states a logarithmic stability estimate of the inclusion encompassing the resonance setting regime.

\begin{theorem}
	\label{thm: stability2}
	
	Let $\Omega \subset \R^3$ satisfy
	\eqref{eqn: domainstart}--\eqref{eqn: domainend} and let $D_1$, $D_2$ be two inclusions contained in $\Omega$ satisfying
	\eqref{eqn: inclstart}--\eqref{eqn: inclend}. 
	Let $\Sigma\subset\partial\Omega$ satisfy \eqref{eqn: ipSigma}.
	Let $\C_{D_1}$ and $\C_{D_2}$ be the isotropic elastic tensors as in \eqref{elastictensorwithinclusion} with $D=D_i$ $i=1,2$, respectively, satisfying  \eqref{eqn: LameCond}, \eqref{eqn: LameUpperBound}, \eqref{eqn: jump}  and let $\rho_{D_1}, \rho_{D_2}$ be the densities as in \eqref{eqn: definition_rho} satisfying \eqref{eqn: rho_bounds}. 
	Let the frequency $\omega$ satisfy \eqref{eqn: omega_bounds}.
	Let $\CC_1$ and $\CC_2$ be the local Cauchy data associated with $D_1, D_2$, respectively. If, for some $\epsilon$, $0<\epsilon<\min\{1,r_0\}$,

	% \begin{equation}
		%   \label{eq:4.1}
		%    \| \Lambda_{D_1}-\Lambda_{D_2}\|_{\mathcal{L}(H^{1/2}(\partial \Omega), H^{-1/2}(\partial
			%    \Omega))}\leq \frac{\epsilon}{r_0},
		% \end{equation}
	\begin{equation}
		d(\CC_1, \CC_2) \leq \frac{\epsilon}{r_0},
	\end{equation}
	then
	\begin{equation}
		\label{eq:4.2}
		d_H(\partial D_1, \partial D_2) \leq r_0 \omega(\epsilon),
	\end{equation}
	where $\omega$ is an
	increasing function on $[0,+\infty)$ satisfying
	\begin{equation}
		\label{eq:4.3}
		\omega(t) \leq C |\log t |^{-\eta}, \ \hbox{for every } \
		0<t<1,
	\end{equation}
	where $C>0$ and $\eta$, with $0<\eta\leq 1$, are constants depending on the a priori data.
\end{theorem}

\section{Singular solutions}

\subsection{Geometrical results}

Let us introduce an augmented domain $\oms$. As in \cite[Section 6]{ARRV}, we can construct an open set $\Omega_0$ exterior to $\Omega$ such that:
% %
\begin{enumerate}
	\item $\Sigma_0=\partial \Omega_0 \cap \partial \Omega\subset\Sigma$;
	\item the size of $\Sigma_0$ is at least $ \frac{r_0}{2}$;
	\item $\oms=\overset{\circ}{\overline{\Omega\cup \Omega_0}}$ is an open and connected set having Lipschitz boundary with constants $r_\sharp$, $M_\sharp$,  $M_\sharp \geq M_0$,
	\begin{equation}
		\label{eq:def-r-sharp}
		r_\sharp=\zeta r_0, \quad 0<\zeta \leq 1,
	\end{equation}
	with $\zeta$ and $M_\sharp$ only depending on $M_0$;
	\item there exists $P_0\in\om_0$ such that
	\begin{equation}\label{eq:ball}
		B_{2r_\sharp}(P_0)\subset\om_0
	\end{equation}
	and, denoting by $S$ the segment joining $P_0$ and $P_\Sigma$, 
	\begin{equation}\label{eq:tubo-P0-Psigma}
		\{ x \in \RR^3 \ | \ \text{dist}(x, S) \leq 2r_\sharp  \} \subset \Omega^\sharp.
	\end{equation}
\end{enumerate}
We extend the elasticity tensor {}from $\Omega$ to $\oms$ defining it as $\C$ in $\om_0\cup \Sigma_0$.

Let $\G$ be the connected component of $\oms\setminus(D_1\cup D_2)$ such that $\p\G \cap \p\oms\neq \emptyset$. Let
\begin{equation}
    \Omega_D = \oms\setminus \overline{\G}.
\end{equation}

Let us introduce as in \cite{AleDiCMorRos14} the modified distance
\begin{equation}
    d_{\mu}(D_1,D_2) = \max \Big\{\max_{x\in \p D_1 \cap \p \Omega_D} dist(x,D_2), \max_{x\in \p D_2 \cap \p \Omega_D} dist(x,D_1) \Big\}.
\end{equation}

\begin{lemma}[Proposition 3.3 in \cite{Al-DC05}]
	\label{lemma-dist-modif}
    Under the assumptions of Theorem \ref{thm: stability2}, there exists a constant $c_0\geq 1$, only depending on the a priori data, such that
    \begin{equation}
        \dH(\p D_1, \p D_2) \leq c_0 d_{\mu}(D_1, D_2).
    \end{equation}
\end{lemma}

For simplicity, let us assume that 
\[
d_{\mu}(D_1,D_2) = \max_{x\in \p D_1 \cap \p \Omega_D} dist(x,D_2).
\]

Let us denote by $C(Q,v,h,\vartheta)$ the closed truncated cone with vertex at $Q$, axis along the direction $v$, height $h$ and aperture $2\vartheta$.

\begin{lemma}[Lemma 4.2 in \cite{AleDiCMorRos14}]\label{lemma 4.2}
    Under the assumptions of Lemma \ref{lemma-dist-modif}, there exist positive constants $\bar{d}$, $c_1$, where $\frac{\bar{d}}{r_0}$ and $c_1$ depend on the a priori data only, and there exists a point $P\in \p D_1$ satisfying
    \begin{equation}
        c_1 d_{\mu}(D_1,D_2) \leq dist(P,D_2)
    \end{equation}
    and such that there exists a path $\gamma$ joining the point $P_0 \in \Omega_0$ introduced above to $P+\overline d \nu$, where $\nu$ is the outer unit normal to $D_1$ at $P$, such that
    \begin{equation}
    	\label{eq:vermeV}
    	V(\gamma) = \bigcup_{S\in \gamma} B_R(S) \cup C\Big(P, \nu, \frac{d^{2}-R^{2}}{d},\arcsin\frac{R}{d} \Big) \subset \Omega^{\sharp}\setminus \Omega_D,
    \end{equation}

    provided $R=\frac{\bar{d}}{\sqrt{1+L_0^{2}}}$ , where $L_0$, only depending on $M_0$ and $\alpha$, is a constant such that $0<L_0\leq M_0$.
\end{lemma}

% \subsection{Alessandrini's identity}
% {\color{blue} da \cite{Beretta2017} p. 9}

% Let us recall an important identity. Let $u_k$ for $k=1,2$ be the weak solution to the following boundary value problem:
% \[
% \begin{cases}
%     \Div(\C_{D_k} \nh u_k) + \rho_{D_k}\omega^{2} u_k = 0 &\text{in }\Omega,\\
%     u_k|_{\p\Omega} \in H^{1/2}(\p\Omega),
% \end{cases}
% \]
% where
% \begin{align*}
%     \C_{D_k}(x) &= \C + (\C^{I} - \C)\chi_{D_k}(x),\\
%     \rho_{D_k}(x) &= \rho + (\ri - \rho)\chi_{D_k}(x).
% \end{align*}
% Then the following version of Alessandrini's identity holds:
% \begin{equation}
% \begin{split}
%     \la (\Lambda_{D_1} - \Lambda_{D_2})u_1|_{\p\Omega}, u_2|_{\p\Omega} \ra = \int_{\Omega} [(\C_{D_1} - \C_{D_2}) \nh u_1 \cdot \nh u_2 - (\rho_{D_1} - \rho_{D_2})\omega^{2} u_1\cdot u_2] \diff x.
% \end{split}
% \end{equation}

\subsection{Green's functions and singular solutions}

\begin{proposition}(\cite[Proposition 5.1]{AleDiCMorRos14}.)\label{eqn: gamma estimates}
Under the assumptions of Theorem \ref{thm: stability2}, for any $y\in \mathbb{R}^3$ there exists a unique fundamental matrix $\Gamma^D(\cdot,y)\in C^0(\mathbb{R}^3\setminus\{y\})$ for the operator $\mbox{div}(\mathbb{C}_D\nabla_x (\cdot))$  in  \eqref{elastictensorwithinclusion} with $\Omega=\mathbb{R}^3$. Moreover, we have 
\begin{equation}\label{simmetria}
    \Gamma^D(x,y)=(\Gamma^D(y,x))^T \ , \ \mbox{for every} \ x\in\mathbb{R}^3, x\neq y \ ,
\end{equation}

\begin{equation}\label{stimaGamma}
    |\Gamma^D(x,y)|\le C |x-y|^{-1} , \ \mbox{for every} \ x\in\mathbb{R}^3, x\neq y\ ,
\end{equation}
\begin{equation}\label{stimagradienteGamma}
    |\nh_x\Gamma^D(x,y)|\le C |x-y|^{-2} , \ \mbox{for every} \ x\in\mathbb{R}^3, x\neq y \ ,
\end{equation}
where $C>0$ depends on the a priori data only.

%\begin{proof} 
%See \cite[Proposition 5.1]{AleDiCMorRos14}.
%\end{proof}

\end{proposition}

Let us denote by $\Gamma$ (Kelvin fundamental solution) and $\Gamma^{+} $ (Rongved fundamental solution) the fundamental matrix $\Gamma^D$ when $D=\emptyset$ and $D=\mathbb{R}^3_+$  respectively .  

%Let $\Gamma$ denote the Kelvin fundamental solution in $\R^{3}$ of the Lamè operator $\Div(\C \nh \cdot)$ with constant coefficients $\lambda$ and $\mu$.
Precisely,  the explicit expression of $\Gamma$ is given by 
\begin{equation}\label{eqn: Kelvin}
	\Gamma(x,y) = \frac{1}{16\pi\mu(1-\nu)}\cdot \frac{1}{|x-y|}\Big(\frac{(x-y)\otimes(x-y)}{|x-y|^{2}} + (3-4\nu)Id \Big)
\end{equation}
where $\nu$ is the Poisson coefficient,  $\nu= \frac{\lambda}{2(\lambda +\mu)}$ .
We refer to \cite[Sect. 10]{AleDiCMorRos14} for the explicit expression of $\Gamma^+$.

\begin{proposition}\label{GreenD1}
    Under the assumptions of Theorem \ref{thm: stability2}, for any $y\in \oms$, there exists a unique Green's function $G_i(\cdot, y)$ continuous in $\oms\setminus \{y\}$ such that, for every $l\in \R^{3}$ with $|l|=1$,

    \begin{equation}\label{eqn:Greeneq}
    	\begin{cases}
    	  \text{div}_x(\C_{D_i}\nh_x G_i(\cdot,y)l) + \rho_{D_i}\omega^{2} G_i(\cdot,y)l = -l\delta(\cdot,y) & \text{in }\oms,\\
    	 G_i(\cdot,y)=0 \ &\text{on }\partial \oms \ , 
    	\end{cases}
    \end{equation}

    with $i=1,2$.
    Furthermore, if $\text{dist}(y,\partial\oms)\geq \gamma r_0$, for some $\gamma>0$ then
    \begin{align}
        \|G_i(\cdot,y) - \Gamma^{D_i}(\cdot,y)\|_{H^{1}(\oms)} &\leq C{r_0}^{-1},\label{eqn: boundDiff}\\
        \|G_i(\cdot,y)\|_{H^{1}(\oms\setminus B_r(y))} &\leq C{r_0^{-1/2}} \ {r^{-1/2}}, \ \ 0<r<r_0\label{eqn: boundGh1}\\
        \|G_i(\cdot,y)\|_{L^{2}(\oms)} &\leq C{r_0}^{-1} \label{eqn: boundGL2}
    \end{align}
    where $C>0$ only depends on the a priori data and on $\gamma$.
\end{proposition}

\begin{proof}
    Set 
    \[
    G_i(x,y)l = \Gamma^{D_i}(x,y)l + w(x,y;l),
    \]
    where $w(\cdot,y;l)\in H^{1}(\oms)$ is the unique solution to

     \begin{equation}\label{eqn:pbw}
    	\begin{cases}
    		   \Div(\C_{D_i}\nh_x w(\cdot,y;l)) +{\rho}_{D_i} \omega^{2} w(\cdot,y;l) = -{\rho}_{D_i} \omega^{2} \Gamma^{D_i}(\cdot,y)l \ &\text{in }\oms,\\
    		w(\cdot,y;l) = -\Gamma^{D_i}(\cdot,y)l \ \ &\text{on }\p\oms,
    	\end{cases}
    \end{equation}

    Hence, $G_{i}$ satisfies \eqref{eqn:Greeneq}.
    
    Furthermore, if $\text{dist}(y,\partial\oms)\geq \gamma r_0$, by \eqref{eqn: omega_bounds},
   \begin{multline*}
    \|w(\cdot,y;l)\|_{H^{1}(\oms)} \leq C(r_0^{-2}\|{\rho}_{D_i} \omega^{2} \Gamma^{D_i}(\cdot,y)\|_{H^{-1}(\oms)} + \|\Gamma^{D_i}(\cdot,y)\|_{H^{1/2}(\p\oms)})\leq\\
    \leq
    C(\|\Gamma^{D_i}(\cdot,y)\|_{H^{-1}(\oms)} + \|\Gamma^{D_i}(\cdot,y)\|_{H^{1/2}(\p\oms)}).
\end{multline*}

    Let us estimate $\|\Gamma^{D_1}(\cdot,y)\|_{H^{1/2}(\p\oms)}$ from above. 
    By the trace theorem, 
    \[
    \|\Gamma^{D_i}(\cdot,y)\|_{H^{1/2}(\p\oms)}\leq C \|\Gamma^{D_i}(\cdot,y)\|_{H^{1}(\oms\setminus B_{\gamma r_0/2}(y))}.
    \]
    By using \eqref{diametro}, it follows that 
    \[
    \begin{split}
    \|\Gamma^{D_i}(\cdot,y)\|_{H^{1}(\oms\setminus B_{\gamma r_0}(y))}^{2} &=r_0^{-3}\left( \|\Gamma^{D_i}(\cdot,y)\|_{L^{2}(\oms\setminus B_{\gamma r_0}(y)))}^{2} + r_0^2\|\nh_x\Gamma^{D_i}(\cdot,y)\|_{L^{2}(\oms\setminus B_{\gamma r_0}(y)))}^{2}\right)\\
    &\leq C r_0^{-3}\left(\int_{\oms\setminus B_{\gamma r_0}(y)} \frac{1}{|x-y|^{2}} dx + r_0^2 \int_{\oms\setminus B_{\gamma r_0}(y)} \frac{1}{|x-y|^{4}} dx \right)\\ 
    &\leq C r_0^{-3}\left(\int_{\gamma r_0}^{M_1 r_0} \frac{1}{\rho^{2}}\rho^{2} d\rho +r_0^2 \int_{\gamma r_0}^{M_1r_0} \frac{1}{\rho^{4}}\rho^{2} d\rho\right)\leq Cr_0^{-2}
    \end{split}
    \]
    Similarly, we have
    
    \[
    \|\Gamma^{D_i}(\cdot,y)\|_{H^{-1}(\oms)} \leq C \|\Gamma^{D_i}(\cdot,y)\|_{L^{2}(\oms)}\leq Cr_0^{-1},
    \]
    so that
      \[
    \|w(\cdot,y;l)\|_{H^{1}(\oms)} \leq Cr_0^{-1}.
    \]
    
    Moreover, we have
    \[
    \begin{split}
    \|G_i(\cdot,y)\|_{H^{1}(\oms\setminus B_r(y))} &\leq \|w(\cdot,y;l)\|_{H^{1}(\oms)} + \|\Gamma^{D_i}(\cdot,y)\|_{H^{1}(\oms\setminus B_r(y))}\\
    &\leq C r_0^{-1}+C r_0^{-1/2}r^{-1/2}\leq C r_0^{-1/2}r^{-1/2}
    \end{split}
    \]
    and, by arguing as above, we obtain \eqref{eqn: boundGL2}.
\end{proof}

%{\color{blue} da \cite{AleDiCMorRos14} p. 2702}

% Given $y\in \R^{3}$ and a concentrated force $l\delta(\cdot-y)$ applied at $y$, with $l\in\R^{3}$ and $|l|=1$, we consider the normalized fundamental solution $u^{D}\in L^{1}_{loc}(\R^{3},\R^{3})$ of the boundary value problem
% \begin{equation}
%     \begin{cases}
%         \text{div}_x(\C_D\nh_x u^{D}(x,y;l)) = -l\delta(x-y) &\text{in }\R^{3}\setminus \{y\}\\
%         \lim\limits_{|x|\rightarrow \infty} u^{D}(x,y;l) = 0
%     \end{cases}
% \end{equation}
% that can be written as
% \begin{equation}
%     \int_{\oms} [\C_D u^{D}(x,y;l)\cdot\nh \varphi - \rho_D\omega^{2} u^{D}(x,y;l)\cdot \varphi] \diff x = l\varphi(y), \qquad \forall \varphi\in C^{\infty}_c(\R^{3},\R^{3}).
% \end{equation}
% It is well-known that
% \[
% u^{D}(x,y;l) = \nh G(\cdot,y)l.
% \]

For any $y,w \in \oms \setminus \Omega_{D}$ and  $l,m\in \R^{3}$  two unitary vectors, i.e., $|l| = |m| =1$, we introduce the singular solution $S$ given by
\begin{equation}\label{eqn: f}
    S(y,w;l,m) = S_{D_1}(y,w;l,m) - S_{D_2}(y,w;l,m),
\end{equation}
where
\begin{equation}
    \begin{split}
        S_{D_1}(y,w;l,m) &= \int_{D_1} (\C^{I} - \C)(x)\nh_x G_1(x,y)l\cdot \nh_x G_2(x,w)m \diff x,\\
        &- \omega^{2} \int_{D_1} (\ri - \rho)(x) G_1(x,y)l\cdot G_2(x,w)m \diff x,\\
        S_{D_2}(y,w;l,m) &= \int_{D_2} (\C^{I} - \C)(x)\nh_x G_1(x,y)l\cdot \nh_x G_2(x,w)m \diff x,\\
        &- \omega^{2} \int_{D_2} (\ri - \rho)(x) G_1(x,y)l\cdot G_2(x,w)m \diff x.
    \end{split}
\end{equation}

Let us fix $y=\overline{y}\in \Omega^{\sharp}\setminus \Omega_{D}$ and $l\in \mathbb{R}^3, |l|=1$. We define 
\begin{eqnarray}\label{vettoref}
	&&\widehat{S}_k(\overline{y}, w;l)=S(\overline{y},w; l ,e_k), \ \ \ k=1,2,3, \nonumber\\
	&& \widehat{S}=(\widehat{S}_1,\widehat{S}_2,\widehat{S}_3).
\end{eqnarray}
Similarly, let us fix $w=\overline{w}$ in $\Omega^{\sharp}\setminus \Omega_{D}$ and $m\in \mathbb{R}^3, |m|=1$. We define 
\begin{eqnarray*}%\label{vettoreftilde}
	&&\widetilde{S}_j(y,\overline{w};m)=S(y,\bar{w};e_j,m), \ \ \ j=1,2,3,\nonumber\\
	&& \widetilde{S}=(\widetilde{S}_1,\widetilde{S}_2,\widetilde{S}_3).
\end{eqnarray*}

%Arguing as in Proposition $4.4$ in \cite{Be-Fr-V} we find the following. 
Given $y=\overline{y}\in \Omega^{\sharp}\setminus \Omega_{D}$ and $l\in \mathbb{R}^3, |l|=1$ and arguing as in Proposition $4.4$ in \cite{Be-Fr-V},
the vector-valued function $\widehat{S}=\widehat{S}(\overline{y},w;l)$ satisfies the Lam\'e system 
\begin{eqnarray}\label{divf}
	\mbox{div}_w (\C\nabla_{w}\widehat{S}) + \rho\omega^2\widehat{S}= 0, \ \ \ \mbox{for every }w \in \Omega^{\sharp}\setminus \Omega_{D}.
\end{eqnarray}

Analogously, choosing $w=\overline{w}\in \Omega^{\sharp}\setminus \Omega_{D}$ and $m\in \mathbb{R}^3, |m|=1$,
the vector-valued function $\widetilde{S}=\widetilde{S}(y,\overline{w};m)$ satisfies the Lam\'e system 
\begin{eqnarray}\label{divftilde}
	\mbox{div}_y (\C\nabla_{y}\widetilde{S})+ \rho\omega^2\widetilde{S}= 0, \ \ \  \mbox{for every }y \in \Omega^{\sharp}\setminus \Omega_{D}.
\end{eqnarray}

By the linearity of $S$ with respect to $l$ and $m$, we have trivially that,
for every $ l,m \in \RR^3, \ |l|=|m|=1$, 
\begin{equation}
	\label{eq:doppia disug per S}
	\begin{aligned}{}
		&
		| S (y,\overline{w};l,m)|
		\leq 
		|\widetilde{S}(y, \overline{w};m)|\leq \sqrt{3} 
		\max_
		{\overset{\scriptstyle l \in \RR^3}{\scriptstyle
				|l|=1}} | S (y,\overline{w};l,m)|,
		\ \ \hbox{for every } y, \, \overline{w} \in \oms \setminus \Omega_{D},
		\\
		&
		| S (\overline{y},w;l,m)|
		\leq 
		|\widehat{S}(\overline{y},w;l)|\leq \sqrt{3} 
		\max_
		{\overset{\scriptstyle m \in \RR^3}{\scriptstyle
				|m|=1}} | {S} (\overline{y},w;l,m)|, 
		\ \ \hbox{for every } \overline{y},w \in \oms \setminus \Omega_{D}.
	\end{aligned}
\end{equation}

\subsection{Upper bound for the singular solutions}

In order to prove the main result, we need to determine how the singular solutions behave near the singularities, namely the boundary of the unknown inclusion. In this Subsection we will show that the singular solution $S$ can be bounded in terms of the power error in the measurements, $\|\Lambda_{D_1}-\Lambda_{D_2}\|_*$. %The tools that we will use are a fundamental regularity result by Li and Nirenberg 

\begin{lemma}( \cite[Theorem 1.1]{LiNUW})\label{lemma:3spheres}
	Let $u \in H^1(B_s(Q))$ be a solution to 
	\begin{equation}
		\label{eq:LBD60.1}
		\begin{aligned}{}
			&
			\dive ( \C \nabla u  ) + \omega^{2} \rho u =0 \quad \hbox{in } B_s(Q),
		\end{aligned}
	\end{equation}
	where the elastic tensor $\C$ satisfies the conditions (iv) in Subsection \ref{Subsec: a-priori}. For every $s_1$, $s_2$, $s_3$, $0<s_1 <s_2<s_3 \leq s$, 
	\begin{equation}
		\label{eq:LBD60.2}
		\begin{aligned}{}
			&
			\|u\|_{L^\infty (B_{s_2}(Q))  } \leq C \|u\|_{L^\infty (B_{s_1}(Q))  }^\tau 
			\|u\|_{L^\infty (B_{s_3}(Q))  }^{1-\tau},
		\end{aligned}
	\end{equation}
	where $C>0$ and $\tau$, $0<\tau<1$, only depend on $\delta_0$, $\gamma_0$, $\lambda$, $\mu$, $ \dfrac{s_2}{s_3}$, $ \dfrac{s_1}{s_3}$.
\end{lemma}

In the following theorem we state the bound from above for the function $S$.

% da inserire qui 

\begin{theorem}\label{upper}
Under the notation of Lemma \ref{lemma 4.2}, let 

\begin{equation}
	\label{eq:23.1}
	y_h = P - h e_3,
\end{equation}
\begin{equation}
	\label{eq:23.2}
	w_h = P -\lambda_w h e_3, \quad 0<\lambda_w < 1,
\end{equation}
with
\begin{equation}
	\label{eq:23.3}
	0<h\leq \overline{d} \left ( 1 - \frac{\sin \widetilde{\vartheta}_0}{4} \right ),
\end{equation}
where $\widetilde{\vartheta}_0 = \arctan \frac{1}{L_0}$ and $n=-e_3$ is the
outer unit normal to $D_1$ at $P$. Then, for every $l$, $m\in\R^3$, $|l|=|m|=1$, we have
\begin{equation}
	\label{eqn:upperBound}
	|S(y_h, w_h; l,m)| \leq \frac{C_0}{\lambda_w h } \epsilon ^{ C_1 \left (
		\frac{h}{r_0}\right)^{C_2}},
\end{equation}
where the positive constants $C_0, C_1, C_2$ only depends on the a priori data.

\end{theorem}

\begin{proof} 
	Let $P_0\in \om_0$, the point satisfying \eqref{eq:ball} and \eqref{eq:tubo-P0-Psigma} with $\text{dist}(P_0,\partial\Omega)>2r_\sharp$.  By the inequality \eqref{stimaCauchy} 
	with the choices $u_1(\cdot)=(G_1(\cdot, y)) l$ and $u_2(\cdot)=(G_2(\cdot, w)) m$ 
	we have that 
	\begin{equation}\label{stimafCauchy}
		|S(y,w,l,m)|\leq d(\CC_1,\CC_2) \|(G_1(\cdot,y),(\C_1\nh G_1(\cdot,y)) n)\|_{\mathcal{H}}
		\cdot \|(G_2(\cdot,w),(\C_2\nh G_2(\cdot,w)) n)\|_{\mathcal{H}}.
	\end{equation}
	By the definition of the $\mathcal{H}$ norm and by \eqref{eqn: boundGh1} we have that 
	\begin{eqnarray}\label{G1}
		&& \|(G_1(\cdot,y),(\C_1\nh G_1(\cdot,y)) n)\|^2_{\mathcal{H}}  \le C \|G_1(\cdot,y)\|^2_{H_{00}^{\frac{1}{2}}(\Sigma)} + C \|(\C_1 \nh G_1(\cdot,y)) n\|^2_{H_{00}^{-\frac{1}{2}}(\partial\Omega)|_{\Sigma}}\le\nonumber\\
		&& \leq C \|G_1(\cdot,y)\|_{H^1(\oms\setminus B_{r_\sharp}(y))}\le C r_0^{-1}r_\sharp^{-1} 
	\end{eqnarray}
	and analogously we have that 
	\begin{eqnarray}\label{G2}
		\|(G_2(\cdot,w),(\C_2\nh G_2(\cdot,w)) n)\|^2_{\mathcal{H}}  \le C r_0^{-1}r_\sharp^{-1} \ .
	\end{eqnarray}
	Hence by gathering \eqref{stimafCauchy}, \eqref{G1} and \eqref{G2} we obtain 
	\begin{equation}\label{stimafCauchy2}
		|S(y,w;l,m)|\leq C	\frac {\epsilon}{r_0} \ \ \ \ \forall \  y,w\in B_{r_\sharp}(P_0),
	\end{equation}
	where $C>0$ is a constant depending on the a priori data only. 
	%The inequality \eqref{stimafCauchy}
	
	%From now we can repeat the previou arguments by substituting the estimate \eqref{step1-f} with \eqref{stimafCauchy2} and the thesis follows. 
	
	The proof continues according to the scheme proposed in \cite{AleDiCMorRos14}[Theorem 6.4].  Namely, by an iterated use of the three spheres inequality for the  solution $\widehat{S}(y, \cdot;l,m)$, first over a chain of balls of constant radius and then over a chain of balls of decreasing radius and tangent to a cone with vertex at the point $P$ introduced in Lemma  \ref{lemma 4.2}. Finally the same iteration is performed for the solution $\tilde{S}(\cdot,w_h;l,m)$ leading to \eqref{eqn:upperBound}.

\end{proof}

\subsection{Lower bound for the singular solutions}

%Reference: \cite[Theorem 6.5]{AleDiCMorRos14}
\begin{theorem}
	Under the notation of Lemma \ref{lemma 4.2}, let $y_h = P-he_3$. For every $i=1,2,3$ there exists $\lambda_w\in \Big\{\frac{2}{3},\frac{3}{4},\frac{4}{5} \Big\}$ and there exists $\overline{h}\in (0,1/2)$ depending only on the a priori data such that 
    \begin{equation}\label{eqn:lowerBound}
        |S(y_h, w_h; e_i,e_i)| \geq \frac{C_3}{h}\qquad \text{with }0<h\leq \overline{h}r,
    \end{equation}
    where
    \begin{align}
        w_h &= P-\lambda_w\,h\,e_3,\\
        r &= \min\Big\{dist(P,D_2),\overline{c}r_0 \Big\}\label{eq: raggio r}
    \end{align}
    with $\bar c=\frac{\min\{1,M_0\}}{12\sqrt{1+M_0^2}}$.
\end{theorem}

%{\color{blue} Equazione (9.3) \cite{AleDiCMorRos14}.}

\begin{proof}
    Without loss of generality, we assume $P\equiv O$ and $e_3=-n$, where $n$ is the outer unit normal to $D_1$ at $O$. Let $h\leq \bar{h}r$ with $r$ defined by \eqref{eq: raggio r} and $\bar{h}\in (0,\frac{1}{2})$ to be chosen later. 
    From \eqref{eqn: f}, we have
    \begin{equation}\label{eqn: triangle_ineq}
    |S(y_h,w_h;l,m)| \geq |S_{D_1}(y_h,w_h;l,m)| - |S_{D_2}(y_h,w_h;l,m)|.
    \end{equation}
    In order to estimate $S_{D_1}$ in \eqref{eqn: triangle_ineq} from below, we write
    \begin{equation}\label{eqn: first plit Sd1}
        \begin{split}
            S_{D_1}&(y_h,w_h;l,m) \\
            &=\int_{D_1\cap B_{r}} (\C^{I} - \C) \nh (\Gamma^{+}(x,y_h)l)\cdot\nh(\Gamma(x,w_h)m)\\
            &+ \int_{D_1\cap B_{r}} (\C^{I} - \C) \nh((G_1(x,y_h))l)\cdot \nh((G_2(x,w_h)-\Gamma(x,w_h))m)\\
            &+ \int_{D_1\cap B_{r}} (\C^{I} - \C)\nh((G_1(x,y_h)-\Gamma^{+}(x,y_h))l)\cdot \nh(\Gamma(x,w_h)m)\\
%            &+ \int_{D_1\cap B_{r}} (\C^{I} - \C)\nh(\Gamma^{+}(x,y_h)l)\cdot \nh((G_2(x,w_h)-\Gamma(x,w_h))m)\\
            &+\int_{D_1\setminus B_{r}} (\C^{I} - \C)\nh(G_1(x,y_h)l)\cdot \nh(G_2(x,w_h)m)\\
            &- \omega^{2}\int_{D_1\cap B_{r}} (\ri - \rho) ((G_1(x,y_h)-\Gamma^{D_1}(x,y_h))l)\cdot (G_2(x,w_h)m)\\
            &- \omega^{2}\int_{D_1\cap B_{r}} (\ri - \rho)(\Gamma^{D_1}(x,y_h)l)\cdot ((G_2(x,w_h)-\Gamma^{D_2}(x,w_h))m)\\
            &- \omega^{2}\int_{D_1\setminus B_{r}} (\ri - \rho)(\Gamma^{D_1}(x,y_h)l)\cdot (\Gamma^{D_2}(x,w_h)m).
        \end{split}
    \end{equation}

    Since the leading term of $S_{D_1}$ as $h\rightarrow 0$ is the first integral on the right-hand side of \eqref{eqn: first plit Sd1}, it is convenient to represent the domain of integration $D_1\cap B_{r}$ as:
    \begin{equation*}
        D_1\cap B_{r} = B^{+}_{r}\cup (D_1\cap B^{-}_{r})\setminus(B^{+}_{r}\setminus D_1).
    \end{equation*}
    We can rewrite $S_{D_1}(y_h,w_h;l,m)$ as:
    \begin{equation}
        S_{D_1}(y_h,w_h;l,m) = I_1 + R_1 + R_2 + R_3 - R_4 - R_5 -R_6,
    \end{equation}
    where
    \begin{align*}
        I_1 &= \int_{B^{+}_{r}} (\C^{I} - \C) \nh (\Gamma^{+}(x,y_h)l)\cdot\nh(\Gamma(x,w_h)m),\\
        R_1 &= \int_{D_1\cap B^{-}_{r}} (\C^{I} - \C) \nh (\Gamma^{+}(x,y_h)l)\cdot\nh(\Gamma(x,w_h)m) \\
        &- \int_{B^{+}_{r}\setminus D_1} (\C^{I} - \C) \nh (\Gamma^{+}(x,y_h)l)\cdot\nh(\Gamma(x,w_h)m),\\
        R_2 &= \int_{D_1\setminus B_{r}} (\C^{I} - \C)\nh(G_1(x,y_h)l)\cdot \nh(G_2(x,w_h)m),\\
        R_3 &= \int_{D_1\cap B_{r}} (\C^{I} - \C) \nh((G_1(x,y_h)-\Gamma^{+}(x,y_h))l)\cdot \nh(\Gamma(x,w_h)m)\\
            &+ \int_{D_1\cap B_{r}} (\C^{I} - \C)\nh(G_1(x,y_h)l)\cdot \nh((G_2(x,w_h)-\Gamma(x,w_h))m)\\
        R_4 &=\omega^{2}\int_{D_1\cap B_{r}} (\ri - \rho) ((G_1(x,y_h)-\Gamma^{D_1}(x,y_h))l)\cdot (G_2(x,w_h)m,\\
              R_5 &= \omega^{2}\int_{D_1\cap B_{r}} (\ri - \rho)(\Gamma^{D_1}(x,y_h)l)\cdot ((G_2(x,w_h)-\Gamma^{D_2}(x,w_h))m),\\
              R_6 &= \omega^{2}\int_{D_1\setminus B_{r}} (\ri - \rho)(\Gamma^{D_1}(x,y_h)l)\cdot (\Gamma^{D_2}(x,w_h)m).
    \end{align*}

    Hence, we have:
    \begin{equation}\label{eqn: lowerSd1}
        |S_{D_1}(y_h,w_h;l,m)| \geq |I_1| - |R_1| - |R_2| - |R_3| - |R_4|- |R_5|- |R_6|.
    \end{equation}
%    Notice that the terms of lower order in $R_4$, $R_5$, $R_6$ have singularities that can be bounded with the same techniques that we will introduce for the gradient. However, since their singularities have a weaker blow-up than the ones of $I_1$, we will focus on the terms $I_1, R_1, R_2$ and $R_3$.

    \begin{itemize}
        \item {Lower bound for $I_1$}: from \cite[Equation (9.20)]{AleDiCMorRos14} we have
        \begin{equation}\label{eqn: estimateI1}
            |I_1| \geq \frac{C_0}{h} - \frac{C_2}{r}.
        \end{equation}
        \item {Upper bound for $R_1$}: by \eqref{stimagradienteGamma} and \eqref{eqn: Kelvin}, and by the change of variables $y=\frac{x}{h}$, we derive
        \begin{equation}
            \begin{split}
                |R_1| &\leq C\int_{\R^{2}} \Big(\int_{-\frac{M_0}{r^{\alpha}_0}|x'|^{1+\alpha}}^{\frac{M_0}{r^{\alpha}_0}|x'|^{1+\alpha}} |x-y_h|^{-2} |x-w_h|^{-2} \diff x_3 \Big) \\
                &\leq C\int_{\R^{2}} \Big(\int_{-\frac{M_0}{r^{\alpha}_0}|x'|^{1+\alpha}}^{\frac{M_0}{r^{\alpha}_0}|x'|^{1+\alpha}} \frac{1}{(|x'|^{2} + (x_3 + h)^{2})(|x'|^{2} + (x_3 + \lambda_w h)^{2})}\diff x_3 \Big) \\
                &= \frac{C}{h} \int_{\R^{2}} \Big(\int_{-\frac{M_0}{r^{\alpha}_0}h^{\alpha}|y'|^{1+\alpha}}^{\frac{M_0}{r^{\alpha}_0}h^{\alpha}|y'|^{1+\alpha}} \frac{1}{(|y'|^{2} + (y_3 + 1)^{2})(|y'|^{2} + (y_3 + \lambda_w)^{2})}\diff y_3 \Big) .
            \end{split}
        \end{equation}
        Following a procedure similar to \cite[pag. 2714]{AleDiCMorRos14} we derive
        \begin{equation}\label{eqn: estimateR1}
            |R_1|\leq \frac{C_1}{r_0}\Big(\frac{h}{r_0} \Big)^{\alpha-1},
        \end{equation}
        where the constant $C>0$ depends on $M_0, \alpha, \alpha_0, \gamma_0, \overline{\lambda}, \overline{\mu}$.
        \item {Upper bound for $R_2$}: since $h\leq \frac{r}{2}$, and $\lambda_w\in (0,1)$, we have that $B_{\frac{\bar h r}{2}}(y_h)\subset B_{r}$ for a certain $\bar h\in (0,1/2)$. Hence, $D_1\setminus B_{r} \subset \oms \setminus B_{\frac{\bar h r}{2}}(y_h)$. By Cauchy-Schwarz inequality it follows that
        \begin{equation*}
            \begin{split}
                |R_2| &\leq \|\C^{I} - \C\|_{L^{\infty}(\oms)} \|\nh G_1(\cdot,y_h)\|_{L^{2}(D_1\setminus B_{r})} \|\nh G_2(\cdot,w_h)\|_{L^{2}(D_1\setminus B_{r})}.
                \end{split}
        \end{equation*}
        By \eqref{eqn: boundGh1}
        \begin{equation*}
                \int_{D_1\setminus B_{r}} |\nh G_1(x,y_h)|^{2} \diff x \leq C \int_{\oms \setminus B_{\frac{\bar h r}{2}}(y_h)} |\nh G_1(x,y_h)|^{2} \diff x \leq C \Big(\frac{\bar h r}{2} \Big)^{-1}.
        \end{equation*}
        Hence, we conclude
        \begin{equation}\label{eqn: estimateR2}
            |R_2| \leq \frac{C_2}{r}.
        \end{equation}
        \item {Upper bound for $R_3$}: We begin by setting
        \begin{equation}
            R_3 = R_3' + R_3''+ R_3'''+ R_3'''',
        \end{equation}
        where
        \begin{align}
            R_3' &= \int_{D_1\cap B_{r}} (\C^{I} - \C) \nh((G_1(x,y_h)-\Gamma^{D_1}(x,y_h))l)\cdot \nh(\Gamma(x,w_h)m),\\
            R_3'' &= \int_{D_1\cap B_{r}} (\C^{I} - \C) \nh((\Gamma^{D_1}(x,y_h)-\Gamma^{+}(x,y_h))l)\cdot \nh(\Gamma(x,w_h)m),\\
            R_3''' &= \int_{D_1\cap B_{r}} (\C^{I} - \C)\nh(G_1(x,y_h)l)\cdot \nh((G_2(x,w_h)-\Gamma^{D_2}(x,w_h))m),\\
            R_3'''' &= \int_{D_1\cap B_{r}} (\C^{I} - \C)\nh(G_1(x,y_h)l)\cdot \nh((\Gamma^{D_2}(x,w_h)-\Gamma(x,w_h))m).
        \end{align}
        For the first quantity, by \eqref{eqn: boundDiff}:
        \begin{multline}\label{eqn: estimateR31}
                |R_3'| \leq Cr_0^{3/2} \|\nh(G_1(\cdot,y_h)-\Gamma^{D_1}(\cdot,y_h))\|_{L^{2}(\oms)} 
                \left(\int_{D_1\cap B_r}|\nh\Gamma(\cdot,w_h)|^2\right)^{1/2}
                \leq\\
                \leq C{r_0}^{-1/2}\left(\int_{D_1\cap B_r}|\nh\Gamma(\cdot,w_h)|^2\right)^{1/2}.
        \end{multline}
        
          Since $x\in D_1\cap B_{r}$, then $|x-w_h|\geq \frac{\lambda_w h}{\sqrt{1+M_0^2}}$, hence $B_{r} \cap D_1 \subset \R^{3}\setminus B_{\frac{\lambda_w h}{{\sqrt{1+M_0^2}}}}(w_h)$. 
        
         We have that
        \begin{equation}\label{stimaGamma}
        	\begin{split}
        	\int_{D_1\cap B_r}|\nh\Gamma(\cdot,w_h)|^2 
        		\leq C \int_{D_1\cap B_{r}} |x-w_h|^{-4} \diff x\le \int_{|x-w_h|\geq \frac{\lambda_w h}{{\sqrt{1+M_0^2}}}} |x-w_h|^{-4} \diff x = \\
        		=\int_{|y|\geq \frac{\lambda_w h}{{\sqrt{1+M_0^2}}}} |y|^{-4} \diff y=  \frac{C}{h}.
        	\end{split}
        \end{equation}
       
       Hence, combining  \eqref{eqn: estimateR31} and \eqref{stimaGamma} we have 
        \begin{equation}\label{eqn: estimateR3first}
       	|R_3'| \leq C r_0^{-1/2} h^{-1/2}.
       \end{equation}
       
       Similarly we have the following estimate
       
        \begin{equation}\label{eqn: estimateR3ter}
        	|R_3'''| \leq Cr_0^{-1/2} h^{-1/2}.
        \end{equation}
        
        The proof of the  estimates for  $R_3''$ and $R_3''''$ are based on the asymptotic   estimates for $\nabla \Gamma^D$ contained in \cite[Section 8]{AleDiCMorRos14} and for brevity we omit the proof. We obtain:
        
        \begin{equation}\label{eqn: estimateR3bis}
        	|R_3''| \leq \frac{C}{r_0}\left(\frac{h}{r_0}\right)^{\gamma -1},
        \end{equation}
        where $\gamma=\frac{\alpha^2}{3\alpha +2}$\,
        \begin{equation}\label{eqn: estimateR3quater}
        	|R_3''''| \leq \frac{C}{r}\ . 
        \end{equation}

\item
{Upper bound for $R_4$}:

By \eqref{eqn: boundDiff}--\eqref{eqn: boundGh1}, we have
              \begin{multline}\label{eqn: estimateR4}
      	|R_4| \leq \frac{C}{r_0^2}\left(\int_{D_1\cap B_r}\left|G_1(x,y_h)-\Gamma^{D_1}(x,y_h)\right|^2\right)^{1/2}
      \left(\int_{ \oms\setminus B_{\frac{\lambda_w h}{\sqrt{1+M_0^2}}}(w_h) }
      \left|G_2(x,w_h)\right|^2\right)^{1/2}\leq\\
      \leq C r_0^{-1/2}h^{-1/2}.
      \end{multline}      
       
 \item
 {Upper bound for $R_5$}:
 
 By \eqref{eqn: boundDiff}--\eqref{stimaGamma}, we have
 \begin{multline}\label{eqn: estimateR4}
 	|R_5| \leq \frac{C}{r_0^2}\left(\int_{D_1\cap B_r}\left|G_2(x,w_h)-\Gamma^{D_2}(x,w_h)\right|^2\right)^{1/2}
 	\left(\int_{B_{r+h}(y_h) }
 	\left|\Gamma^{D_1}(x,y_h)\right|^2\right)^{1/2}
 	\leq C r_0^{-1}.
 \end{multline}

   \item
  {Upper bound for $R_6$}:
  
  By \eqref{stimaGamma}, we have
  \begin{multline}\label{eqn: estimateR4}
  	|R_6| \leq \frac{C}{r_0^2}
  	\left(\int_{B_{r+h}(y_h) }
  	\left|\Gamma^{D_1}(x,y_h)\right|^2\right)^{1/2}
  	\left(\int_{B_{r+h}(w_h) }
  	\left|\Gamma^{D_2}(x,w_h)\right|^2\right)^{1/2}
  	\leq C r_0^{-1}.
  \end{multline}

    Collecting together \eqref{eqn: lowerSd1}, \eqref{eqn: estimateI1}, \eqref{eqn: estimateR1}, \eqref{eqn: estimateR2}, \eqref{eqn: estimateR3first}--\eqref{eqn: estimateR4}, we have
    \begin{multline}
    	\label{eqn: lowerestimateS_D1}
        |S_{D_1}(y_h,w_h;l,m)| \geq \frac{C_0}{h} - \frac{C_1}{r_0}\Big(\frac{h}{r_0} \Big)^{\alpha-1} - \frac{C_2}{r} - C_3 r_0^{-1/2}h^{-1/2} 
        - \frac{C_1}{r_0}\Big(\frac{h}{r_0} \Big)^{\gamma-1} \geq\\
        \geq \frac{\tilde C}{h} \Big(1 - \frac{h}{r} 
        - \frac{h}{r_0} - \left(\frac{h}{r_0}\right)^\gamma -
        \left(\frac{h}{r_0}\right)^\alpha  - \left(\frac{h}{r_0}\right)^{1/2} \Big).
    \end{multline}

\item
 {Upper bound for $S_{D_2}$}: 

Let us estimate {}from above the first addend of $S_{D_2}(y_h,w_h;l,m)$. 
Recalling that $D_2\cap B_r=\emptyset$ and $h<\frac{r}{2}$, for every $x\in D_2$ we have   $|x-y_h|\geq |x|-h\geq r-h>\frac{r}{2}$, $|x-w_h|\geq |x|-h\geq r-h>\frac{r}{2}$. Therefore, by applying Schwarz inequality and \eqref{eqn: boundGh1} we have that

\begin{multline}
	\label{stimaS_D2}
	\left|\int_{D_2}(\C^I-\C)(x)\nh G_1(x,y_h)l\cdot \nh G_2(x,w_h)m  \right|\leq\\
	\leq C\left(\int_{\oms\setminus B_{r/2}(y_h)}|\nh  G_1(x,y_h)|^2 \right)^{1/2} \left(\int_{\oms\setminus B_{r/2}(y_h)}|\nh  G_2(x,w_h)|^2 \right)^{1/2}\leq
	\frac{C}{r},
\end{multline}

\begin{multline}
	\label{stimaS_D2}
	\left|\omega^2\int_{D_2}(\rho^I-\rho) G_1(x,y_h)l\cdot G_2(x,w_h)m \right|\leq\\
	\leq \frac{C}{r_0^2}\left(\int_{\oms\setminus B_{r/2}(y_h)}|  G_1(x,y_h)|^2 \right)^{1/2} \left(\int_{\oms\setminus B_{r/2}(y_h)}|  G_2(x,w_h)|^2 \right)^{1/2}\leq
	\frac{C}{r}
\end{multline}
\end{itemize}

By \eqref{eqn: lowerestimateS_D1} and \eqref{stimaS_D2}, it follows that there exists $\bar h$ only depending on the a priori data, such that, for every $h$, $0<h\leq\bar hr$, estimate \eqref{eqn:lowerBound} holds.
    
\end{proof}

\subsection{Proof of Theorem \ref{thm: stability2}}

\begin{proof}
	Let $P$ be the point introduced in Lemma   \ref{lemma 4.2}. Let us recall the definition \eqref{eq: raggio r}:
    \[
    r =\min\{dist(P,D_2), \bar{c}r_0\}.
    \]

    By combining the upper bound \eqref{eqn:upperBound}, with $l=m=e_i$ for $i\in \{1,2,3\}$, and the lower bound \eqref{eqn:lowerBound} for the singular solution $S$ as defined in equation \eqref{eqn: f}, we obtain that there exists
    $\lambda_w\in\{\frac{2}{3},\frac{3}{4},\frac{4}{5}\}$ such that
    \begin{equation}\label{eqn:combining}
    \frac{C_3}{h}\leq |S(y_h,w_h; e_i, e_i)|\leq \frac{C_0}{\lambda_w h}\epsilon^{C_1 \left(\frac{h}{r_0}\right)^{C_2}}\leq \frac{3C_0}{2 h}\epsilon^{C_1 \left(\frac{h}{r_0}\right)^{C_2}}\quad \text{for every } h,\,\,0<h\leq\bar{h}{r}.
    \end{equation}
    Taking the logarithm of the inequality \eqref{eqn:combining}, and recalling that $\varepsilon\in (0,1)$, we deduce
    \begin{equation}
        h\leq Cr_0 \Big(\frac{1}{|\log\varepsilon|}\Big)^{\frac{1}{C_2}}.
    \end{equation}
    In particular, choosing $h=\bar{h}r$ yields
    \[
    r \leq C r_0\Big( \frac{1}{|\log\varepsilon|}\Big)^{\frac{1}{C_2}}.
    \]
    Now, consider the two possible cases for $r$:
    \begin{itemize}
        \item If $r=dist(P,D_2)$, the results follows from \cite[Lemma 4.1 and 4.2]{AleDiCMorRos14}.
        \item If $r = \bar{c}r_0$, the results follows by observing that $d_h(\p D_1, \p D_2)\leq diam(\Omega)\leq Cr_0$, where $C$ depends on $M_0, M_1$.
    \end{itemize}
    
    \end{proof}

\section{The stability result for the local Dirichlet to Neumann map }
In this section, under an additional spectral hypothesis
on $\omega$, we prove the continuous dependence of the inclusion in terms of the local Dirichlet-to-Neumann map.

Precisely, let us define
\[
\lambda^{0}_1 = \min_{u\in H^{1}_0(\Omega)} r_0^2\cdot
\dfrac{\int_{\Omega} \C_0 \nh u \cdot \nh u\, }{\int_{\Omega} \gamma_0^{-1} |u|^2\, },
\]
where $\C_0$ is the constant linear tensor with components
\begin{equation}
	(C_0)_{ijkl} = \Big(\frac{\beta_0 - 2\overline{\mu}}{3} \Big)\delta_{ij} \delta_{kl} + \alpha_0 (\delta_{ki} \delta_{lj} + \delta_{li} \delta_{kj}).
\end{equation}
Notice that, by a scaling argument, $\lambda^{0}_1$ does not depend on $r_0$.

By our choice of $\C_0$ and by \eqref{eqn: rho_bounds}, we have that
\begin{equation}\label{princeigen}
	\frac{\lambda^{0}_1}{r_0^2}\leq
	\min_{u\in H^{1}_0(\Omega)} 
	\dfrac{\int_{\Omega} \C_D \nh u \cdot \nh u\, }{\int_{\Omega} \rho_D |u|^2\,} ,
\end{equation}
where the right hand side in \eqref{princeigen} is the principal eigenvalue of the operator $-\Div( \C_D\nh u)-\rho_D u$ in $\Omega$. The general theory of elliptic equations 
guarantees the well-posedness of the boundary value problem
\begin{equation}\label{eqn: standardElastic}
	\begin{cases}
		\operatorname{div}(\C_{D}\nh u)+\rho_D\omega^2 u=g &\text{in}~\Omega,\\
		u=f&\text{on}~\partial\Omega ,
	\end{cases}
\end{equation}
under the hypothesis 
\begin{equation}\label{spectral}
	0<\omega<  \frac{\sqrt{\lambda^{0}_1}}{r_0} \ .
\end{equation}

\begin{proposition}\label{prop: wellposedness}
	Let $\Omega\subset \R^{3}$ be a bounded domain with Lipschitz boundary with constants $r_0, M_0$ and let $D$ be a subdomain of $\Omega$ . Let $g\in H^{-1}(\Omega)$ and $f\in H^{1/2}(\partial\Omega)$. 
	Let $\C_{D}\in L^{\infty}(\Omega)$ be an isotropic elasticity tensor as in  \eqref{elastictensorwithinclusion} with elastic material satisfying the properties in $(iv)$ and $\rho_{D}\in  L^{\infty}(\Omega)$ as in \eqref{eqn: definition_rho}
	satisfying \eqref{eqn: rho_bounds}.
	Then, for any $\omega$ satisfying \eqref{spectral}
	there exists a unique solution of the boundary value problem \eqref{eqn: standardElastic}
	satisfying
	\begin{equation}
		\|u\|_{H^{1}(\Omega)} \leq C(\|f\|_{H^{1/2}(\partial \Omega)} + r_0^2\|g\|_{H^{-1}(\Omega)})
	\end{equation}
	where $C$ depends on $\alpha_0, \beta_0, \gamma_0$ and $\lambda^{0}_1$.
\end{proposition}

We consider the local Dirichlet-to-Neumann map
% %
\begin{equation}\label{eqn: DtoN}
	\Lambda_D^{\Sigma}: H^{1/2}_{00}(\Sigma) \rightarrow
	H^{-1/2}_{00}(\Sigma),
\end{equation}
which can be characterised as the bilinear form on $H^{1/2}_{00}(\Sigma) \times H^{1/2}_{00}(\Sigma)$
\begin{equation}
	\la \Lambda_D^{\Sigma} f, g\ra =\int_\Omega \C_D \nh u \cdot \nh v - \rho_D \omega^{2} u\,\cdot\,v,
\end{equation}
for all $f,g\in H^{1/2}_{00}(\Sigma)$, where $u$ solves \eqref{eqn: standardElastic}, $v$ is any  $H^1(\Omega)$ function such that $v=g$ on $\p\Omega$ in the trace sense and $\omega$ satisfies \eqref{spectral}.

We denote with $\|\cdot\|_*$ the operator norm in $\mathcal{L}(H^{1/2}_{00}(\Sigma),H^{-1/2}_{00}(\Sigma))$.

\begin{theorem}\label{thm: stability}
	Let $\Omega \subset \R^3$ satisfy
	\eqref{eqn: domainstart}--\eqref{eqn: domainend} and let $D_1$, $D_2$ be two inclusions contained in $\Omega$ satisfying
	\eqref{eqn: inclstart}--\eqref{eqn: inclend}. 
	Let $\Sigma\subset\partial\Omega$ satisfy \eqref{eqn: ipSigma}.
	Let $\C_{D_1}$ and $\C_{D_2}$ be the isotropic elastic tensors as in \eqref{elastictensorwithinclusion} with $D=D_i$ $i=1,2$, respectively, satisfying  \eqref{eqn: LameCond}, \eqref{eqn: LameUpperBound}, \eqref{eqn: jump}  and let $\rho_{D_1}, \rho_{D_2}$ be the densities as in \eqref{eqn: definition_rho} satisfying \eqref{eqn: rho_bounds}. 
	Let the frequency $\omega$ satisfy \eqref{spectral}.
	
	Let $\Lambda^{\Sigma}_{D_1}$ and $\Lambda^{\Sigma}_{D_2}$ be the local Dirichlet-to-Neumann maps associated with $D_1, D_2$, respectively. If, for some $\epsilon$, $0<\epsilon<1$,
	\begin{equation}\label{eq:4.1}
		\| \Lambda_{D_1}-\Lambda_{D_2}\|_{\mathcal{L}(H^{1/2}(\partial \Omega), H^{-1/2}(\partial
			\Omega))}\leq \frac{\epsilon}{r_0},
	\end{equation}
	then
	\begin{equation}
		\label{eq:4.2}
		d_H(\partial D_1, \partial D_2) \leq r_0 \omega(\epsilon),
	\end{equation}
	where $\omega$ is an
	increasing function on $[0,+\infty)$ satisfying
	\begin{equation}
		\label{eq:4.3}
		\omega(t) \leq C |\log t |^{-\eta}, \ \hbox{for every } \
		0<t<1,
	\end{equation}
	where $C>0$ and $\eta$, with $0<\eta\leq 1$, are constants depending on the a priori data.
\end{theorem}

\begin{proof}
Similarly to Theorem \ref{thm: stability2}, the proof is derived by comparing upper and lower bounds for the singular solutions near the boundary of the unknown inclusion.
The difference consists in expressing the upper bound in terms of the norm of the difference of the local DtoN maps.
To this aim, we recall the Alessandrini's identity \cite{A1} for the time-harmonic case: for $G_1(\cdot,y)l,\, G_2(\cdot,w)m\in H^1(\Omega)$,
\[
\big\langle (\Lambda_{D_1}^{\Sigma} - \Lambda_{D_2}^{\Sigma})\, G_1(\cdot,y)l,\; G_2(\cdot,w)m \big\rangle 
= S(y,w;l,m).
\]
for every $y,w\in B_{r_\sharp}(P_0)$, with $l,m\in \R^{3}$, $|l|=|m|=1$. Hence, by standard trace estimates and by Proposition \ref{GreenD1} it follows that 
\begin{equation}\label{step1-f}
	|S(y,w;l,m)| \le C \frac{\epsilon}{r_0}, \qquad \text{for all } y,w\in B_{r_\sharp}(P_0),
\end{equation}
where $C$ depends only on the a priori data.

From now on, we can repeat the previous arguments by substituting the estimate \eqref{stimafCauchy2} with \eqref{step1-f}. 

\end{proof}

% \begin{remark}
	%   \label{remark:localmap}
	%   In the case when $D_1$, $D_2$ are at a prescribed positive distance {}from $\partial\Omega$, it is also  possible to obtain a result analogous to the above Theorem when the
	%   Dirichlet-to-Neumann maps $\Lambda_{D_1}$, $\Lambda_{D_2}$ are replaced with local maps.
	%   For instance, fixing $Q\in\partial\Omega$ and given $\rho_1>0$, denoting
	%   $\Sigma=\partial\Omega\cap B_{\rho_1}(Q)$, we introduce
	%   \begin{equation*}
		%   H_{co}^{1/2}(\Sigma)=\{g\in H^{1/2}(\partial\Omega)\ |\ \hbox{supp}\ g\subset\subset\Sigma\}
		%   \end{equation*}
	%    and define
	%    \begin{equation*}
		%    \Lambda^\Sigma_{D_i}: H_{co}^{1/2}(\Sigma)\rightarrow(H_{co}^{1/2}(\Sigma))^*\subset
		%   H^{-1/2}(\partial\Omega)
		%   \end{equation*}
	%   as the restriction of $\Lambda_{D_i}$ to $H_{co}^{1/2}(\Sigma)$.
	%   Thus, replacing the assumption \eqref{eq:4.1} with
	%   \begin{equation*}
		%   \| \Lambda_{D_1}^\Sigma-\Lambda_{D_2}^\Sigma\|_{\mathcal{L}\left(H^{1/2}_{co}(\Sigma),
			%   \left(H^{1/2}_{co}(\Sigma)\right)^*\right)}\leq \frac{\epsilon}{r_0},
		%   \end{equation*}
	%   we obtain \eqref{eq:4.2}--\eqref{eq:4.3} with constants only depending on the a-priori data
	%    and on $\rho_1$. Such a result is a nearly straightforward adaptation of the theory developed in \cite{AK12}.
	%    \end{remark}

\section*{Acknowledgments}
The work of AM was supported by PRIN 2022 n. 2022JMSP2J ”Stability, multiaxial fatigue
and fatigue life prediction in statics and dynamics of innovative structural and material coupled systems”. AM is a member of the INdAM Gruppo Nazionale di Fisica Matematica.
SF research was funded by the Austrian Science Fund (FWF) SFB 10.55776/F68: "Tomography Across the Scales", project F68-01.
The work of ES  was supported by the Italian MUR through the PRIN 2022 project ``Inverse problems in PDE: theoretical and numerical analysis'', project code 2022B32J5C, under the National Recovery and Resilience Plan (PNRR), Italy, funded by the European Union  - Next Generation EU, Mission 4 Component 1 CUP F53D23002710006.
ES has also been supported by Gruppo Nazionale per l'Analisi \text{Matematica,} la Probabilità e le loro applicazioni (GNAMPA) by the grant "Equazioni alle Derivate Parziali: problemi inversi e stime quantitative".

\bibliographystyle{plain}
\bibliography{references}{}

\end{document}

%% file: packages.tex
\usepackage{geometry}
\usepackage[utf8]{inputenc}
\usepackage[T1]{fontenc}
\usepackage{lmodern}
\usepackage[english]{babel}
\usepackage{charter}
\usepackage{enumitem} 

\usepackage{amsmath,amsfonts,amssymb,amsthm}

\usepackage{graphicx, multicol,color}

\usepackage{hyperref}

\theoremstyle{plain}
\newtheorem{theorem}{Theorem}[section]
\newtheorem{lemma}[theorem]{Lemma}
\newtheorem{proposition}[theorem]{Proposition}

\theoremstyle{definition}
\newtheorem{definition}[theorem]{Definition}

\theoremstyle{remark}
\newtheorem{remark}[theorem]{Remark}

\DeclareMathOperator{\Div}{div}

\newcommand{\R}{\mathbb{R}}

\newcommand{\N}{\mathbb{N}}
\newcommand{\C}{\mathbb{C}}

\newcommand{\M}{\mathbb{M}}

\def\XXint#1#2#3{{\setbox0=\hbox{$#1{#2#3}{\int}$}
     \vcenter{\hbox{$#2#3$}}\kern-.5\wd0}}

\numberwithin{equation}{section}

\newcommand{\diff}{\hspace{0.3em}\mathrm{d}}
\newcommand {\p} {\partial}

\newcommand{\F}{\mathcal{F}}
\newcommand{\G}{\mathcal{G}}
\newcommand{\la}{\langle}
\newcommand{\ra}{\rangle}

\newcommand{\CC}{\mathcal{C}}
\newcommand{\dH}{\mathrm{d}_{\mathcal{H}}}

\newcommand{\ri}{\rho^{I}}
\newcommand{\nh}{\widehat\nabla}
\newcommand{\hash}{\textrm{\#}}
\newcommand{\oms}{\Omega^{\hash}}

\title{ {\bf Identification of an inclusion {}from local Cauchy data for time-harmonic elastic waves}}

\author{
    Sonia Foschiatti\thanks{Faculty of Mathematics, University of Vienna, Austria.}\and
    Antonino Morassi\thanks{Dipartimento Politecnico di Ingegneria e Architettura, Universit\`{a} degli Studi di Udine, Italy.}\and 
     Edi Rosset\thanks{Dipartimento di Matematica, Informatica e Geoscienze, Universit\`{a} degli Studi di Trieste, Italy.} \and
    Eva Sincich\footnotemark[3]
}

\date{}